\documentclass[11pt]{amsart}
\usepackage{geometry}                
\usepackage{amssymb}
\usepackage{multirow}
\usepackage{mathrsfs}
\usepackage{longtable}

\numberwithin{equation}{section}
\numberwithin{figure}{section}
\theoremstyle{plain}
\newtheorem{thm}{Theorem}[section]
\theoremstyle{definition}
\newtheorem{defn}[thm]{Definition}

\newtheorem*{claim}{Claim}
\theoremstyle{plain}

\theoremstyle{remark}
\newtheorem{rem}[thm]{Remark}
\theoremstyle{plain}
\newtheorem{cor}[thm]{Corollary}
\theoremstyle{plain}
\newtheorem{lem}[thm]{Lemma}
\theoremstyle{remark}
\newtheorem*{rem*}{Remark}

\newcommand{\abs}[1]{\left\vert#1\right\vert}
\newcommand{\set}[1]{\left\{#1\right\}}

\newcommand{\Cpx}{\mathbb{C}}
\newcommand{\Qua}{\mathbb{H}}

\newcommand{\half}{\frac{1}{2}}
\newcommand{\sph}{\mathbb{S}}

\newcommand{\Hil}{\mathbf{S}}
\newcommand{\fts}{\footnotesize}

\newcommand{\dive}{\mathrm{div}}
\newcommand{\Ric}{\mathrm{Ric}}
\newcommand{\Tr}{\mathrm{tr}}

\newcommand{\vol}{\mathrm{vol}}
\newcommand{\Hess}{\mathrm{Hess}}

\newcommand{\sfG}{\mathsf{G}}
\newcommand{\sfK}{\mathsf{K}}

\newcommand{\sfSU}{\mathsf{SU}}
\newcommand{\sfU}{\mathsf{U}}
\newcommand{\sfSO}{\mathsf{SO}}
\newcommand{\Spin}{\mathsf{Spin}}
\newcommand{\sfSp}{\mathsf{Sp}}
\newcommand{\sfE}{\mathsf{E}}
\newcommand{\sfF}{\mathsf{F}}

\title[Linear stability of $\nu$-entropy on symmetric spaces]{Linear stability of Perelman's $\nu$-entropy \\ on symmetric spaces of compact type}
\author{Huai-Dong Cao}
\address{
Huai-Dong Cao \\
Department of Mathematics, Lehigh University \\
Bethlehem, PA, 18015}
\email{huc2@lehigh.edu}
\author{Chenxu He}
\address{
Chenxu He \\
Department of Mathematics, University of Oklahoma \\
Norman, OK, 73019}
\email{che@math.ou.edu}

\subjclass[2000]{53C25, 53C35, 53C44. \\
The research of the first author was partially supported by NSF Grant DMS-0909581.}

\begin{document}

\maketitle

\begin{abstract}
Following \cite{CaoHamiltonIlmanen}, in this paper we study the linear stability of Perelman's $\nu$-entropy on Einstein manifolds with positive Ricci curvature. We observe the equivalence between the linear stability restricted to the transversal traceless symmetric 2-tensors and the stability of Einstein manifolds with respect to the Hilbert action. As a main application, we give a full classification of linear stability of the $\nu$-entropy on symmetric spaces of compact type.  In particular, we exhibit many more linearly stable and linearly unstable examples than previously known and also the first  linearly stable examples, other than the standard spheres, whose second variations are negative definite. 
\end{abstract}

\section{Introduction}

Hamilton's Ricci flow can be regarded as a dynamical system on the space of Riemannian metrics over a smooth manifold $M^n$ modulo diffeomorphisms and scalings. Einstein metrics,  or more generally Ricci solitons,  can be viewed as fixed points in the system. When $M^n$ is compact, G. Perelman \cite{Perelman}  introduced the $\mathcal{W}$-functional 
\[
\mathcal{W} (g, f, \tau) = \int_M \left[ \tau \left(R + \abs{\nabla f}^2\right) + f - n \right] (4\pi \tau)^{-\frac{n}{2}} e^{-f}dV,
\]
where $g$ is a Riemannian metric, $R$ the scalar curvature, $f$ a smooth function on $M$ and $\tau$ a positive scale parameter. The associated $\nu$-entropy is defined by 
\[
\nu(g) = \inf \set{\mathcal{W}(g,f, \tau) : f\in C^{\infty}(M), \tau > 0, (4\pi \tau)^{-\frac{n}{2}}\int_M e^{-f}dV = 1}.
\]
A remarkable fact of the $\nu$-entropy showed by Perelman is that it is monotone increasing under the Ricci flow, and its critical points are precisely shrinking gradient Ricci solitons defined by the equation
\[ 
\Ric+\nabla^2 f=\frac {1} {2\tau} g
\] 
for some potential function $f$ on $M$.  Clearly, when potential functions are constant shrinking gradient Ricci solitons reduce to positive Einstein manifolds.

 In \cite{CaoHamiltonIlmanen} Hamilton, Ilmanen and the first author initiated the study of linearly stability of Einstein metrics. In particular, they found the second variation formula of the $\nu$-entropy for Einstein manifolds with positive Ricci curvature and investigated the linear stability of certain Einstein manifolds. An Einstein metric is called \emph{linearly stable} if the second variation of the $\nu$-entropy is semi-negative definite. Otherwise it is called \emph{linearly unstable}. They showed that though the standard spheres and complex projective spaces are linearly stable, many known Einstein manifolds are unstable along the Ricci flow so that a generic perturbation around these Einstein metrics acquire higher $\nu$-entropy and hence the flow can never return near the original metrics. 

On the other hand, Einstein metrics of volume $1$ on a closed manifold can be characterized variationally as the critical points of the Hilbert action \cite{Hilbert}, which associates each Riemannian metric  $g$ of volume $1$ the integral of its scalar curvature:
\begin{equation*}
\Hil(g) = \int_M R_g dV.
\end{equation*}
The gradient vector of the Hilbert action with respect to the natural $L^2$ metric is precisely the negative of the traceless Ricci tensor. This variational approach of finding Einstein metrics, for example on compact homogeneous spaces, has been studied by C. B\"{o}hm, M. Wang and W. Ziller in \cite{BoehmWangZiller}. It is well-known that the second variation of the Hilbert action on the standard spheres vanishes under certain conformal deformation. Other than this case, when the Einstein constant is positive, the integral of scalar curvature is increasing if the metric is deformed conformally with unit volume. It raises an interesting question for the stability of the Hilbert action when the metric is deformed in other directions. The standard spheres and the complex projective spaces are examples where the integral is non-increasing when the metrics are perturbed in directions transversal to conformal variations. 

We denote by $C^\infty(S^2(T^*M))$ the space of all covariant symmetric $2$-tensors on $M$.  When $M$ is a compact Einstein manifold other than the standard sphere, we have the following decomposition \cite[Lemma 4.57]{Besse}
\begin{equation}\label{eqn:S2Mdecomposition}
C^{\infty}(S^2(T^*M)) = \mathrm{Im} \dive^* \oplus \mathscr{C}M \cdot g \oplus \left(\ker \dive \cap \ker \Tr\right)
\end{equation} 
where $\mathscr{C}M$ denotes the set of smooth functions on $M$, $\dive$ and $\Tr$ stand for divergence and trace operators respectively. Since the infinitesimal deformation of a Riemannian metric defines a covariant symmetric $2$-tensor on $M$, it is natural to consider the stability on each factor.
The first factor represents deformations by diffeomorphisms. Since the $\nu$-entropy and the Hilbert action are invariant under the diffeomorphisms, the second variations vanish on this factor. The second factor represents conformal deformations and the third factor consists of all transversal traceless symmetric $2$-tensors. Note that $h\in C^\infty(S^2(T^*M))$ is called \emph{transversal} if its divergence vanishes. The second variation formula of the Hilbert action on an Einstein metric restricted to the last two factors is well-known, see for example Theorem 4.60 in \cite{Besse}. The second variation of Perelman's $\nu$-entropy on the second and third factors has been discussed in \cite{CaoHamiltonIlmanen}. More precisely, for $g(t) = g + t h$ with $h \in C^{\infty}(S^2(T^* M))$ write the second variation as 
\[
\delta^2 \nu_g (h,h) = \frac{d^2}{dt^2}\Big{|}_{t=0} \nu(g+t h) = \frac{1}{2\lambda \vol(M, g)}\int_M \langle Nh, h\rangle dV
\]
where $\lambda>0$ is the Einstein constant and  $N$ is a certain self-adjoint operator acting on symmetric $2$-tensors which is closely related to the Lichnerowicz Laplacian, see the formula of $N$ in section 3 or \cite{CaoHamiltonIlmanen}. Then we have the following

\begin{thm}[Cao-Hamilton-Ilmanen]\label{thm:CHI2ndvariation}
Let $(M^n, g)$ be a compact Einstein manifold other than the standard sphere, with $\Ric = \lambda g$ and $\lambda>0$. Then the decomposition (\ref{eqn:S2Mdecomposition}) is orthogonal with respect to the second variation $\delta^2 \nu_g$ of $\nu$-entropy.  Moreover 
\begin{enumerate}
\item the first factor $\mathrm{Im}\dive^*$ is contained in the null space of $N$; 
\smallskip
\item $\delta^2 \nu_g$ is positive for some $h \in \mathscr{C}M \cdot g$ if and only if the first eigenvalue $\mu_{\text{fns}}$ of Laplacian on functions satisfies
\[
 -2\lambda   < \mu_{\text{fns}} <-\frac{n}{n-1}\lambda ;
\]  
\item for any transversal traceless symmetric $2$-tensor $h$ we have
\[
\delta^2\nu_g(h,h) = \half \int_M \langle (\Delta_L + 2\lambda) h, h\rangle dV,
\]
where $\Delta_L h = \Delta h + 2Rm(h, \cdot)-\Ric \circ h-h\circ \Ric$  is the Lichnerowicz Laplacian acting on symmetric $2$-tensors. 
\end{enumerate}
\end{thm} 
\begin{rem}
(1) Note that the bound  $$ \mu_{\text{fns}}<-\frac{n}{n-1}\lambda$$ follows from the Lichnerowicz bound of the first eigenvalue of Laplacian on positive Einstein manifolds and the rigidity theorem by M. Obata, see \cite{Lichnerowicz} and \cite{Obata}.  \\
(2) The results in this theorem are essentially contained in \cite{CaoHamiltonIlmanen}. Here we use the decomposition (\ref{eqn:S2Mdecomposition}) of symmetric $2$-tensors to formulate it in the way above.   
\end{rem}

We denote by $\mu_L$ the maximal eigenvalue of the Lichnerowicz Laplacian on transversal traceless symmetric $2$-tensor. A linearly stable Ricci soliton is called \emph{neutrally linearly stable} if the second variation vanishes for some nonzero $h$ not in $\mathrm{Im}\dive^*$. Theorem \ref{thm:CHI2ndvariation} has the following 

\begin{cor}
Let $(M, g)$ be  an Einstein manifold with Ricci curvature $\lambda > 0$ other than the standard sphere. Then $g$ is linearly stable with respect to Perelman's $\nu$-entropy if and only if $\mu_{\text{fns}} \leq -2\lambda$ and $\mu_L \leq -2 \lambda$. It is neutrally linearly stable if it is linearly stable and either $\mu_{\text{fns}}=-2\lambda$ or $\mu_L=-2\lambda$. 
\end{cor}

It is quite interesting to observe that, when restricted to transversal traceless tensors, the second variations of $\nu$-entropy and the Hilbert action are exactly the same. It is also quite remarkable that the bound $-2\lambda$ of $\mu_{\text{fns}}$ appears in another notion of stability, the identity map regarded as a harmonic map, see \cite{Smith}. The stability of Einstein metric with respect to the Hilbert action has been studied by many people, and the case when the Einstein manifolds are simply-connected symmetric spaces of compact type was classified by N. Koiso in \cite{KoisocompactSS} 
with a few exceptions.  

One of our main results in this paper is a full classification of stability of Perelman's $\nu$-entropy on symmetric spaces of compact type. Since the stability of $\nu$-entropy is also related to the first eigenvalue $\mu_{\text{fns}}$ by Theorem 1.1, we are able to deal with the exceptional examples in Koiso's classification. Below we list the classification in low dimensions, say $\dim M \leq 10$, and leave out the standard spheres which are known to be linearly stable. The full classification is given by Theorem \ref{thm:classificationcompactSS} in section 4(see also Tables \ref{tab:eigenvaluesgroups} $\&$  \ref{tab:eigenvaluesGK} in Appendix A).  

\begin{thm}\label{thm:lowdim}
Let $(M^n,g)$,  $n\le 10$, be a simply-connected irreducible symmetric spaces of compact type other than the standard sphere. Then the linear stability with respect to Perelman's $\nu$-entropy is given by the following table, where ``n. l. stable" stands for ``neutrally linearly stable". 

\begin{center}
\begin{tabular}{|p{4cm}| c | p{2cm} || p{3.5cm} |c| p{2.5cm} |}
\hline
$\quad \quad \quad M$ & $n$ & $\quad$stability & $\quad \quad \quad M$ & $n$ & $\quad$ stability \\
\hline 
\hline
$\Cpx \mathrm{P}^m$\fts{($m=2,3,4, 5$)} & $2m$ & \small{n. l. stable} & $\frac{\sfSp(2)}{\sfU(2)} = \frac{\sfSO(5)}{\sfSO(3)\times \sfSO(2)}$ & $6$ & \small{linearly unstable} \\
\hline
$\frac{\sfU(4)}{\sfU(2)\times \sfU(2)} = \frac{\sfSO(6)}{\sfSO(4)\times \sfSO(2)}$ & $8$ & \small{n. l. stable} & $\sfSp(3)/\sfSp(1)\times \sfSp(2)$ & $8$ & \small{linearly unstable} \\
\hline
$\sfSU(3)/\sfSO(3)$ & $5$ & \small{n. l. stable} & $\sfSU(3)$ & $8$ & \small{linearly unstable} \\
\hline
$\frac{\sfSU(4)}{\sfSO(4)} = \frac{\sfSO(6)}{\sfSO(3)\times \sfSO(3)}$ & $9$ & \small{n. l. stable} & $\Spin(5)$ & $10$ & \small{linearly unstable} \\
\hline
$\sfSO(7)/\sfSO(5)\times \sfSO(2)$ & $10$ & \small{n. l. stable} & $\sfG_2/\sfSO(4)$ & $8$ & \small{linearly stable} \\
\hline
\end{tabular}
\end{center} 
\end{thm}

\begin{rem}
The examples of complex projective spaces $\Cpx \mathrm{P}^m$ with any $m\geq 2$, complex hyperquadrics $Q^4 = \sfSO(6)/\sfSO(4)\times \sfSO(2)$ and $Q^3 = \sfSO(5)/\sfSO(3)\times \sfSO(2)$ have been discussed in \cite{CaoHamiltonIlmanen}. The second variation of $\nu$-entropy on $\sfG_2/\sfSO(4)$ is negative definite which gives the first such example other than the standard spheres.
\end{rem}

\begin{rem}
The first eigenvalue of Laplacian on functions and Lichnerowicz Laplacian on transversal traceless are also important for the stability of black holes and event horizons in physics.  For example, in the Freund-Rubin compactification, the stability condition of an Einstein manifold $M^n$ with Einstein constant $\lambda > 0$ is given by
\[
\mu_L \geq - \frac{\lambda}{n-1}\left(4 - \frac{1}{4}(n-5)^2\right).
\]
From the estimates of $\mu_\text{fns}$ and $\mu_L$ in \cite{GibbonsHartnoll}, \cite{GibbonsHartnollPope}, \cite{GubserMitra}, \cite{PagePope1} and \cite{PagePope2}, the following examples are linearly unstable: 
\begin{enumerate}
\item The three infinite families of homogeneous Einstein metrics in dimensions $5$ and $7$ in \cite{Romans}, \cite{CastellaniDAuriaFre}, \cite{DAuriaFreVan} and  \cite{PagePope1}.  They are $\sph^1$-bundles over $\sph^2 \times \sph^2$, $\Cpx \mathrm{P}^2 \times \sph^2$ and $\sph^2\times \sph^2 \times \sph^2$ respectively. These examples are special cases of the examples in \cite{WangZiller} by M. Wang and W. Ziller. 
\item A few of the inhomogeneous Einstein metrics on the products of spheres in low dimensions constructed by C. B\"{o}hm in \cite{Boehm}.
\end{enumerate}
\end{rem}


\smallskip

The paper is organized as follows. In section 2 we set up the convention of curvature tensors, Laplacians and other geometric quantities, and then collect a few useful facts of the second variation of the Hilbert action. In section 3 we recall the second variation formula of the $\nu$-entropy obtained by \cite{CaoHamiltonIlmanen} and then give a detailed proof of Theorem \ref{thm:CHI2ndvariation}. In section 4, we first quote Koiso's results on stability of symmetric spaces with respect to Hilbert action and then prove our classification result, Theorem \ref{thm:classificationcompactSS}. Finally, in Appendix A, we list the first eigenvalues of Laplacian on functions and Lichnerowicz Laplacian on symmetric $2$-tensors for symmetric spaces of compact type.

\medskip

\textbf{Acknowledgment}. The second named author would like to thank Wolfgang Ziller for helpful communications. 

\medskip{}

\section{Preliminaries}
In this section we first fix our conventions of Riemann curvature tensors, Laplacians, etc. In the second part we collect some useful facts about stability of Einstein metrics with respect to Hilbert action and, more details can be found in \cite[Chapter 4]{Besse} and the references therein.

Let $(M^n , g)$ be a Riemannian manifold. At any fixed point $x\in M$ let $\set{E_i}_{i=1}^n$ be a geodesic frame, i.e., $g(E_i, E_j) = \delta_{ij}$ and $\nabla_{E_i} E_j|_x = 0$. We denote by $\nabla_i = \nabla_{E_i}$ the covariant derivative. The Riemann curvature tensor is given by $R_{ijkl} = g(Rm(E_i, E_j)E_k, E_l)$ and on the round sphere we have $R_{ijji}\geq 0$. It follows that the Ricci curvature is given by 
\[
\Ric(E_i, E_j) = R_{ij} = \sum_{k} R_{ikkj}. 
\]
The covariant derivative commuting formula, for example for a covariant $2$-tensor $\beta_{kl}$, is given by
\[
\nabla_i \nabla_j \beta_{kl} - \nabla_j \nabla_i \beta_{kl} = - \sum_{p} R_{ijkp}\beta_{pl} -\sum_{p} R_{ijlp} \beta_{kp}.
\]

For any covariant  $2$-tensor $h \in C^{\infty}(T^*M \otimes T^*M)$ its divergence is given by 
\begin{equation*}
(\dive h)_i = - (\delta h)_{i} = \sum_{j=1}^n \nabla_j h_{ji}.
\end{equation*}
$h$ is called \emph{transversal} if $\dive h = 0$. Let $f\in C^{\infty}(M)$ be a smooth function. We denote the Laplacian
\[
\Delta f = \sum_{i=1}^{n}\nabla_i \nabla_i f.
\]
The rough Laplacian acting on tensors is also denoted by $\Delta$. For example,  on covariant $2$-tensors it is given by
\begin{equation*}
\left(\Delta h\right)_{ij} = \sum_{k =1}^n \nabla_k \nabla_k h_{ij}.
\end{equation*}
Note that $C^\infty(S^2(T^* M))$ denotes all covariant symmetric $2$-tensors on $M$. The linear map
\[
Rm(h, \cdot)_{ij} = \sum_{k,l=1}^{n}R_{iklj}h_{kl}
\] 
maps $C^\infty(S^2(T^*M))$ to itself. The Lichnerowicz Laplacian on $C^\infty(S^2(T^*M))$ is defined by 
\[
(\Delta_L h)_{ij} = (\Delta h)_{ij} + 2 \sum_{k,l=1}^{n}R_{iklj} h_{kl} - \sum_{k=1}^{n} \left(R_{ik}h_{kj} + h_{ik}R_{kj}\right).
\]
If $(M,g)$ is an Einstein manifold with $\Ric = \lambda g$, then we have
\[
\Delta_L h = \Delta h + 2Rm(h, \cdot) - 2\lambda h. 
\]
Note that our Laplacian and Lichnerowicz Laplacian are different from those in \cite{Besse} and \cite{KoisocompactSS} by a minus sign. If $M$ is compact, then our Laplacians are semi-negative definite. 

Next we collect a few facts of stability of Einstein metric with respect to Hilbert action. Recall that the Hilbert action of the Riemannian metric $g$ with volume $1$ is the following functional
\[
\Hil (g) = \int_M R_g dV,
\]
where $R_g$ is the scalar curvature of $(M, g)$. The critical points of this action are precisely Einstein metrics. Let $\mathcal{M}_1$ be the set of all Riemannian metrics with volume $1$ on $M$. Suppose that $g(t) \in \mathcal{M}_1$ is a one-parameter family with $g(0) = g$ an Einstein metric, then the first variation formula of volume implies that
\[
\int_M\left( \Tr_g h\right) dV  = 0, 
\] 
where 
\[
h = \frac{d}{dt}\Big{|}_{t=0} g(t)
\]
is a symmetric $2$-tensor. The second variation formula of $\Hil$ at $g$ along the direction $h$ is given in Proposition 4.55 in \cite{Besse}. Using the decomposition (\ref{eqn:S2Mdecomposition}) of symmetric $2$-tensors on Einstein manifolds we have

\begin{thm}[Theorem 4.60 in \cite{Besse}]
Let $(M^n, g)$ be a compact Einstein manifold other than the standard sphere. Then the decomposition 
\[
T_g \mathcal{M}_1 = \mathrm{Im} \dive^* \oplus \mathscr{C}_g M \cdot g \oplus \left(\ker \dive \cap \ker \Tr\right)
\]
is orthogonal with respect to the Hessian $\Hil''_g$, where 
\[
\mathscr{C}_gM = \set{f\in C^\infty (M) : \int_M f dV_g = 0}
\]
and 
\[
\ker \dive \cap \ker \Tr = \set{h \in C^{\infty}(S^2 (T^*M)) : \dive h = 0 \, \, \text{and}\,\, \Tr h = 0}.
\]
Furthermore we have
\begin{enumerate}
\item the first factor is contained in the null-space of $\Hil''_g$;
\item the second variation on the second factor is given by 
\begin{equation*}
\Hil ''_g(f g, f g) = - \frac{n-2}{2}\int_M f \left[(n-1)\Delta_g f + R_g f\right] dV \quad \text{for any}\quad  f\in \mathscr{C}_g(M);  
\end{equation*}
\item for $h \in \ker \dive\cap \ker \Tr$ we have
\begin{equation*}
\Hil''_g(h, h) = \half \int_M \left(h \Delta h + 2 Rm(h,h)\right) dV = \half \int_M \left(h \Delta_L h + 2 \lambda \abs{h}^2\right) dV
\end{equation*}
where $\lambda$ is the Einstein constant of $(M, g)$.
\end{enumerate}
In particular the restriction of $\Hil''_g$ to the second factor is positive definite; the nullity and the coindex of the restriction of $\Hil''_g$ to the third factor are finite. 
\end{thm}
Note that the coindex is the number of the positive eigenvalues of $\Hil''_g$.  

\begin{defn}[Definition 2.7 in \cite{KoisocompactSS}]
An Einstein metric $g$ on $M$ with $\Ric = \lambda g$ is \emph{stable} if $\Hil''_g$ is negative definite on $\ker \dive\cap \ker\Tr$, i.e., for any transversal traceless symmetric $2$-tensor $h$ we have
\[
\int_M \left(h \Delta h + 2Rm(h,h)\right)dV = \int_M \left(h \Delta_L h + 2\lambda \abs{h}^2\right)dV < 0.
\]
It is called \emph{unstable} if $\Hil''_g$ has positive eigenvalue on $\ker\dive \cap \ker \Tr$.
\end{defn}
\begin{rem}
Note that in \cite[Definition 4.63]{Besse}, $g$ is called stable if $\Hil ''_g$ is uniformly negative definite, i.e., there exists $\kappa > 0$ such that  
\[
\int_M \left(h \Delta h + 2Rm(h,h)\right)dV = \int_M \left(h \Delta_L h + 2\lambda \abs{h}^2\right)dV \leq - \kappa \int_M \abs{h}^2 dV
\]
for any nonzero $h \in \ker \dive \cap \ker \Tr$. Since the space $\ker \dive\cap \ker \Tr$ is infinite dimensional, it is stronger than the definition above of \cite{KoisocompactSS}. 
\end{rem}

\medskip{}

\section{The second variation of Perelman's $\nu$-entropy on Einstein metrics}

In this section we first recall the second variation formula of Perelman's $\nu$-entropy on Einstein manifolds by Hamilton, Ilmanen and the first author. Then from the decomposition (\ref{eqn:S2Mdecomposition}) of symmetric $2$-tensors on Einstein manifolds, we analyze the second variation on each factor and prove Theorem \ref{thm:CHI2ndvariation}.

For a $1$-from $\omega$ we denote by
\[
(\dive^* \omega)_{ij} = - \half\left(\nabla_i \omega_j + \nabla_j \omega_i \right) = - \half \left(\mathscr{L}_{\omega^\sharp} g\right)_{ij},
\]
where $\omega^\sharp$ is the dual vector field of $\omega$ given by 
\[
g(\omega^\sharp, X) = \omega(X) \quad \text{for any}\quad X \in TM.
\]
The second variation of $\nu$-entropy at a positive Einstein metric is given by the following 
\begin{thm}[Cao-Hamilton-Ilmanen \cite{CaoHamiltonIlmanen}]
Let $(M^n, g)$ be an Einstein manifold with Einstein constant $\frac{1}{2\tau}> 0$ and consider variations $g(s) = g + s h$. Then the second variation $\delta_g^2 \nu (h, h)$ is given by
\[
\frac{d^2}{ds^2}\Big{|}_{s=0} \nu(g(s)) = \frac{\tau}{\vol(M, g)}\int_M \langle N h, h\rangle dV,
\]
where
\[
N h = \half \Delta h + Rm(h,\cdot) + \dive^* \dive h + \half \Hess v_h - \frac{g}{2n \tau \vol(M, g)} \int_M \Tr_g h dV, 
\]
and $v_h$ is the unique solution of 
\[
\Delta v_h + \frac{v_h}{2\tau} = \dive \dive h.
\]
\end{thm}
\begin{rem}
(1) The details of the calculation of the second variation formula at a general shrinking Ricci soliton are given by the first author and M. Zhu in \cite{CaoZhu} which includes the above formula as a special case.  \\
(2) Note that the uniqueness of $v_h$ follows from the Lichnerowicz bound of the first eigenvalue of Laplacian. Integrating both sides yields
\[
\int_M v_h dV= 0
\]
by the divergence theorem on closed manifold. 
\end{rem}

\begin{defn}
An Einstein metric $(M^n, g)$ is called \emph{linearly stable} with respect to $\nu$-entropy if the second variation $\delta^2_g\nu (h,h) \leq 0$ for any $h \in C^\infty(S^2(T^* M))$. Otherwise it is called \emph{linearly unstable}. A linearly stable Einstein metric is called \emph{neutrally linearly stable} if $\delta^2_g \nu(h, h) = 0$ for some nonzero $h \in C^\infty(S^2(T^*M))$ not in $\mathrm{Im}\dive^*$.
\end{defn}

In the rest of the paper we assume that the Einstein metric $(M, g)$ has  $\Ric = \lambda g$. Since Perelman's $\nu$-entropy is diffeomorphism invariant  we have $\delta^2 \nu_g(h, h) = 0$ for any $h \in \mathrm{Im}\dive^*$. We first show that the stability operator $N$ actually vanishes on $\mathrm{Im}\dive^*$. 
 
\begin{lem}\label{lem:Nimdiv}
Let $\xi\in TM$ be a vector field and $h = 2 \mathscr{L}_{\xi} g$. Then we have $N h = 0$.
\end{lem}
\begin{proof}
Let $\xi_i = g(\xi, E_i)$ be the components of $\xi$. Then we have
\[
h_{ij} = 2\left(\mathscr{L}_{\xi} g\right)_{ij} = \nabla_i \xi_j + \nabla_j \xi_i \quad \text{for}\quad 1\leq i, j\leq n. 
\]
\begin{claim}
The unique solution $v_h$ is given by $v_h = \Tr h = 2\dive \xi$.
\end{claim}
First note that 
\begin{eqnarray*}
\nabla_i \nabla_j \nabla_i \xi_j & = & \nabla_j \nabla_i \nabla_i \xi_j - R_{ijip}\nabla_p \xi_j - R_{ijjp}\nabla_i \xi_p \\
& = & \nabla_j \nabla_i \nabla_i \xi_j + R_{jp}\nabla_p \xi_j - R_{ip} \nabla_i \xi_p \\
& = & \nabla_j \nabla_i \nabla_i \xi_j.
\end{eqnarray*}
Then we have
\begin{eqnarray*}
\dive \dive h - \Delta (\Tr h) & = & \nabla_j \nabla_i h_{ij} - \nabla_j \nabla_j h_{ii} \\
& = & \nabla_j \nabla_i \nabla_i \xi_j + \nabla_j \nabla_i \nabla_j \xi_i - 2\nabla_j \nabla_j \nabla_i \xi_i \\
& = & \nabla_j \left(\nabla_i \nabla_j \xi_i - \nabla_j \nabla_i \xi_i \right) + \nabla_j \nabla_i \nabla_j \xi_i - \nabla_j \nabla_j \nabla_i \xi_i \\
& = & \nabla_j (- R_{ijik} \xi_k) + \left(\nabla_i \nabla_j \nabla_i \xi_j - R_{jiik}\nabla_k \xi_j - R_{jijk}\nabla_i\xi_k\right) - \nabla_i \nabla_i \nabla_j \xi_j \\
& = & (\nabla_j R_{jiik})\xi_k + R_{jiik}\nabla_j \xi_k - R_{jk}\nabla_k \xi_j + R_{ik}\nabla_i\xi_k + \nabla_i \nabla_j \nabla_i \xi_j - \nabla_i \nabla_i \nabla_j \xi_j \\
& = & R_{jk}\nabla_j \xi_k + \nabla_i \nabla_j \nabla_i \xi_j - \nabla_i \nabla_i \nabla_j \xi_j \\
& = & \lambda \nabla_i \xi_i + \nabla_i\left(-R_{jijk}\xi_k\right) \\
& = & \lambda \nabla_i \xi_i + R_{ik}\nabla_i \xi_k \\
& = & 2\lambda \nabla_i \xi_i,
\end{eqnarray*}
i.e.,
\[
\Delta (\Tr h) + \lambda \Tr h = \dive \dive h,
\]
and this shows that $v_h = \Tr h$. 

Next we compute the $2$-tensor $\dive^* \dive h$. First we have
\begin{eqnarray*}
-2 (\dive^* \dive h)_{jk} & = & \nabla_j \nabla_i h_{ik} + \nabla_k \nabla_i h_{ij} \\
& = & \nabla_j \nabla_i \nabla_i \xi_k + \nabla_j \nabla_i \nabla_k \xi_i + \nabla_k \nabla_i \nabla_i \xi_j + \nabla_k \nabla_i \nabla_j \xi_i \\
& = & \nabla_i \nabla_j \nabla_i \xi_k - R_{jiip}\nabla_p \xi_k - R_{jikp}\nabla_i\xi_p + \nabla_j\left(\nabla_k \nabla_i \xi_i  - R_{ikip}\xi_p\right) \\
& & + \nabla_i \nabla_k \nabla_i \xi_j - R_{kiip}\nabla_p \xi_j - R_{kijp}\nabla_i \xi_p + \nabla_k\left(\nabla_j \nabla_i \xi_i - R_{ijip}\xi_p\right).
\end{eqnarray*}
Note that $\nabla_j \left(R_{ikip}\xi_p\right) = - R_{kp} \nabla_j \xi_p$ as $\Ric = \lambda g$. We rearrange the terms and obtain
\begin{eqnarray*}
-2 (\dive^* \dive h)_{jk} & = & \nabla_i \nabla_j \nabla_i \xi_k + \nabla_j \nabla_k \nabla_i \xi_i + \nabla_i \nabla_k \nabla_i \xi_j + \nabla_k \nabla_j \nabla_i \xi_i \\
& & - R_{jp}\nabla_p \xi_k + R_{kp}\nabla_j \xi_p - R_{kp}\nabla_p \xi_j + R_{jp}\nabla_k \xi_p - R_{jikp}\nabla_i \xi_p - R_{kijp}\nabla_i\xi_p \\
& = & \nabla_i \nabla_j \nabla_i \xi_k + \nabla_j \nabla_k \nabla_i \xi_i + \nabla_i \nabla_k \nabla_i \xi_j + \nabla_k \nabla_j \nabla_i \xi_i \\
& & - R_{jikp}\nabla_i \xi_p - R_{kijp}\nabla_i \xi_p \\
& = & \nabla_i \nabla_i \nabla_j \xi_k - \nabla_i\left(R_{jikp}\xi_p\right) + 2 \nabla_j \nabla_k \left(\nabla_i \xi_i\right) + \nabla_i \nabla_i \nabla_k \xi_j - \nabla_i\left(R_{kijp}\xi_p\right) \\
& &  - R_{jikp}\nabla_i \xi_p - R_{kijp}\nabla_i \xi_p \\
& = & \Delta h_{ij} + \nabla_j \nabla_k v_h - R_{jikp}\nabla_i \xi_p - R_{kijp}\nabla_i \xi_p - R_{jikp}\nabla_i \xi_p - R_{kijp}\nabla_i \xi_p.
\end{eqnarray*}
For the last equation we used that $\nabla_i R_{jikp} = \nabla_i R_{kijp} = 0$ since $\Ric$ is parallel. We evaluate
\begin{eqnarray*}
Rm(h, \cdot)_{jk} & = & R_{jipk}h_{ip} = R_{jipk}\nabla_i \xi_p + R_{jipk}\nabla_p \xi_i \\
& = & R_{jipk}\nabla_i \xi_p + R_{jpik}\nabla_i \xi_p \\
& = & - R_{jikp}\nabla_i \xi_p - R_{kijp}\nabla_i \xi_p.
\end{eqnarray*}
It follows that
\[
-2 \left(\dive^* \dive h\right)_{jk} = \Delta h_{jk} + \nabla_j \nabla_k v_h + 2Rm(h,\cdot)_{jk},
\]
i.e.,
\begin{equation*}
\frac{1}{2}\Delta h + \dive^* \dive h + \frac{1}{2}\Hess v_h + Rm(h, \cdot) = 0.
\end{equation*}
So we have $Nh = 0$ as $\int_M \Tr h dV = 0$. This finishes the proof.
\end{proof}

The following lemma of the stability along conformal variations is stated in \cite{CaoHamiltonIlmanen}, and the outline of the proof is also provided there. We provide the details of the proof for the convenience of the reader. 

\begin{lem}[Cao-Hamilton-Ilmanen]\label{lem:stabilityconformal}
Suppose that $(M^n, g)$ is an Einstein manifold with Ricci curvature $\lambda > 0$ other than the standard sphere. The second variation of the $\nu$-entropy on conformal variations is unstable if and only if the first (nonzero) eigenvalue of Laplacian on functions $\mu_{\text{fns}}$ satisfies
\[
-2\lambda < \mu_{\text{fns}} <   -\frac{n}{n-1}\lambda.
\]
\end{lem}
\begin{proof}
By Remark 1.2 (1), it suffices to show the bound $-2\lambda  < \mu_{\text{fns}}$. 
For any smooth function $f\in \mathscr{C}M$, since  $\Delta + \lambda$ is invertible  there exists a function $u$ which is the unique solution to the equation
\[
\Delta u + \lambda u = f.
\]
Following \cite{CaoHamiltonIlmanen}, for any function $u \in \mathscr{C}M$ let
\[
S(u) = (\Delta u) g - \mathrm{Hess} u + u \lambda g.
\]
Then, by direct computations, one can check that the symmetric $2$-tensor $S(u)$ is divergence free and 
\[
\Delta_L (S(u))=S(\Delta u). 
\]
Since $\Hess u \in \mathrm{Im}\dive^*$, we have 
\[
N (fg) = N\left((\Delta u + \lambda u) g\right) = N(S(u)))
\]
and
\begin{eqnarray*}
\int_M \Tr S(u) dV & = & \int_M \left((n-1)\Delta u + n \lambda u \right) dV = n \int_M \lambda u dV \\
& = & n \int_M f dV.
\end{eqnarray*}
Hence, we have 
\begin{eqnarray*}
N(f g) & = & N(S(u))\\
& = & \half \Delta (S(u)) + Rm(S(u), \cdot) - \frac{\lambda g}{n \vol(g)}\int_M \Tr S(u) dV \\
& = & \half \Delta_L (S(u)) + \lambda S(u) - \frac{\lambda \int_M f dV}{\vol(g)} g \\
& = & \half S(\Delta u)+ \lambda S(u) - \frac{\lambda \int_M f dV}{\vol(g)} g. 
\end{eqnarray*}
Note that in the last equality we used $\Delta_L (S(u)) = S(\Delta u)$. In the following we assume that $\vol(g) =1$ and write 
\[
u = \sum_{i=0}^{\infty} a_i \phi_i
\]
where $\set{\phi_i}_{i=0}^{\infty}$ are eigenfunctions of $\Delta$ that give a $L^2$-orthonormal basis of $\mathscr{C}M$. Let $k_i$ be the eigenvalue of $\Delta$ on $\phi_i$ with $k_0 = 0$. It follows that 
\begin{eqnarray*}
f = \sum_i (k_i + \lambda) a_i \phi_i, & & \int_M f dV = \lambda a_0; \\
\Delta u = \sum_i k_i a_i \phi_i, & & \Delta(\Delta u) = \sum_i k_i^2 a_i \phi_i.
\end{eqnarray*}
So we have
\begin{eqnarray*}
\Tr S(\Delta u) + 2\lambda \Tr S(u) & = & \left((n-1)\Delta (\Delta u) + n \lambda \Delta u\right) +  2\lambda\left((n-1)\Delta u + n \lambda u\right) \\
& = & \sum_{i}\left((n-1)k_i + n \lambda \right) (k_i + 2\lambda) a_i \phi_i.
\end{eqnarray*}
It follows that 
\begin{eqnarray*}
2\int_M g(N (fg), fg) dV & = & \int_M f\sum_i \left((n-1)k_i + n \lambda \right) (k_i + 2\lambda) a_i \phi_i dV - 2n\lambda^3 a_0^2 \\
& = & \sum_i (k_i + \lambda)\left((n-1)k_i + n \lambda \right) (k_i + 2\lambda) a_i^2 - 2n\lambda^3 a_0^2 \\
& = & \sum_{i\geq 1} (k_i + \lambda)((n-1)k_i + n \lambda )(k_i + 2\lambda)a_i^2.
\end{eqnarray*}
Since $(M, g)$ is not the standard sphere, we have $(k_i + \lambda)((n-1)k_i+ n \lambda) > 0$. It follows that $\int_M g(N (fg),fg)dV > 0$ for some $f$ if and only if $k_1 + 2\lambda > 0$, i.e., the second variation in conformal directions is unstable if and only if the first nonzero eigenvalue $\mu_{\text{fns}}$ of Laplacian on functions is strictly bounded below by $-2\lambda$.
\end{proof}

Using the previous lemmas, we can now prove Theorem \ref{thm:CHI2ndvariation} stated in the Introduction.

\begin{proof}[Proof of Theorem \ref{thm:CHI2ndvariation}]
From Lemmas \ref{lem:Nimdiv} and \ref{lem:stabilityconformal} we only need to show that the second factor $\mathscr{C}M \cdot g$ and the third factor $\ker \dive \cap \ker \Tr$ are orthogonal to each other with respect to the bi-linear form $\int_M \langle N \cdot, \cdot \rangle dV$. Suppose that $h \in \ker \dive \cap \ker \Tr$ and $ u \in C^\infty(M)$. Then we have
\[
N h = \half \Delta h + Rm(h, \cdot)
\] 
and  
\begin{eqnarray*}
\int_M \langle N h, u g\rangle dV & = & \int_M \left(\half g (u g, \Delta h) + Rm(h, u g)\right) dV \\
& = & \int_M \half  g(h, \Delta (u g)) dV + \int_M u g(\Ric , h) dV \\
& = & \half \int_M \Delta u (\Tr h) dV + \int_M \lambda u (\Tr h) dV \\
& = & 0.
\end{eqnarray*}
This finishes the proof of the theorem.
\end{proof}

\medskip{}

\section{The linear stability of symmetric spaces of compact type}

In \cite{CaoHamiltonIlmanen} the linear stability of several examples of compact Einstein manifolds with positive Ricci curvature has been studied. In this section we prove the full classification of linear stability on symmetric spaces of compact type, see Theorem \ref{thm:classificationcompactSS}.

By the work of Hamilton, Ilmanen and the first author \cite{CaoHamiltonIlmanen}, it is known that any product metric of two positive Einstein manifolds (of the same Einstein constant) is unstable, thus 
we may restrict our attention to irreducible ones. We also exclude the case when the manifold is the standard sphere,  as it is geometrically stable from the results of R. Hamilton in \cite{Hamilton3d,Hamilton4d,Hamilton2d} and G. Huisken in \cite{Huisken}. 

Recall that the stability operator restricted to transversal traceless symmetric $2$-tensors is given by 
\[
N h = \half \Delta h + Rm(h, \cdot) = \half \left(\Delta_L h + 2\lambda h\right) \quad \text{for any }  h \in \ker\dive\cap \ker \Tr.
\]
The stability of Einstein metrics on symmetric spaces $(\sfG/\sfK, g)$ with respect to Hilbert action has been studied by N. Koiso. When $\sfG/\sfK$ is of compact type, from his results in \cite{KoisocompactSS} and \cite{Koisorigidity} we have the following 
\begin{thm}[Koiso]\label{thm:Koiso}
Let $M = \sfG/\sfK$ be a simply-connected irreducible symmetric space of compact type other than the standard sphere. Then we have 
\begin{enumerate}
\item the stability operator of $M$ has zero eigenvalue in $\ker\dive\cap \ker \Tr$ if $M$ is one of the following spaces:
\begin{eqnarray*}
& \sfSU(n+1)(n \geq 2), \quad \sfU(p+q)/\sfU(p)\times \sfU(q)(p\geq q\geq 2), & \\
& \sfSU(n)/\sfSO(n) (n\geq 3), \quad \sfSU(2n)/\sfSp(n)(n\geq 3), \quad \sfE_6/\sfF_4; & 
\end{eqnarray*}
\item the stability operator of $M$ has positive eigenvalue in $\ker \dive\cap \ker \Tr$, i.e., $M$ is unstable with respect to the Hilbert action, if $M$ is one of the following spaces:
\[
\sfSp(n)(n \geq 2) \quad {\mbox or} \quad \sfSp(n)/\sfU(n)(n\geq 3);
\]
\item the stability operator of $M$ is negative definite in $\ker\dive\cap \ker \Tr$, i.e., $M$ is stable with respect to Hilbert action, if $M$ is not in Cases (1), (2) or the following spaces:
\[
\sfSp(p+q)/\sfSp(p)\times \sfSp(q)(p=2, q=1\text{ or }p\geq q\geq 2), \quad \sfF_4/\Spin(9).
\]
\end{enumerate}
\end{thm}
\begin{rem}
Note that from \cite{KoisocompactSS}, the first eigenvalue $\mu_L$ on $\sfSO(5)/\sfSO(3)\times \sfSO(2)$ is only bounded above by $-\frac{4}{3}\lambda$. However, it was shown by Gasqui and Goldschmidt \cite{GasquiGoldschmidt} that $\mu_L = -\frac{4}{3}\lambda$, which implies that the symmetric metric on $\sfSO(5)/\sfSO(3)\times \sfSO(2)$ is unstable with respect to the Hilbert action. 
\end{rem}

Combining Koiso's results in Theorem 4.1, Remark 4.2, and the first eigenvalues $\mu_{\text{fns}}$ of simply-connected irreducible symmetric spaces $\sfG/\sfK$ of compact type, see \cite{Urakawa} and \cite{Nagano}, we have the following classification results. 

\begin{thm}\label{thm:classificationcompactSS}
Let $\sfG/\sfK$ be a simply-connected irreducible symmetric space of compact type other than the standard sphere. Then we have
\begin{enumerate}
\item  $\sfG/\sfK$ is linearly unstable with respect to Perelman's $\nu$-entropy if it is one of the followings:
\begin{enumerate}
\item the simple Lie groups $\sfSU(n+1)$($n \geq 2$), $\sfSp(n)$($n \geq 2$), 
\item the complex hyperquadric $\sfSO(5)/\sfSO(3)\times \sfSO(2)$ and the quaternionic Grassmannians $\sfSp(p+q)/\sfSp(p)\times \sfSp(q)$($p\geq q \geq 1$, $p+q\geq 3$), 
\item the Cayley projective plane $\sfF_4/\Spin(9)$ and the spaces $\sfSU(2n)/\sfSp(n)$($n\geq 3$), $\sfSp(n)/\sfU(n)$($n\geq 3$), $\sfE_6/\sfF_4$;
\end{enumerate}
\item $\sfG/\sfK$ is neutrally linearly stable with respect to Perelman's $\nu$-entropy if it is one of the followings:
\begin{enumerate}
\item the simple Lie group $\sfG_2$,
\item the complex Grassmannians $\sfU(p+q)/\sfU(p)\times \sfU(q)$($p\geq q \geq 1, p+q\geq 3$),
\item the complex hyperquadrics $\sfSO(n+2)/\sfSO(n)\times \sfSO(2)$($n\geq 5$), 
\item the spaces $\sfSU(n)/\sfSO(n)$($n\geq 3$), $\sfSO(2n)/\sfU(n)$($n\geq 5$), $\sfE_6/\sfSO(10)\cdot \sfSO(2)$ and $\sfE_7/\sfE_6\cdot \sfSO(2)$;
\end{enumerate}
Furthermore, except $\sfSU(n)/\sfSO(n)$ they have neutrally stable deformation along conformal directions and $\sfU(p+q)/\sfU(p)\times \sfU(q)$ also have neutrally stable deformation along traceless transversal symmetric $2$-tensors;
\item $\sfG/\sfK$ is linearly stable without neutrally stable deformation with respect to Perelman's $\nu$-entropy, if it is not in the previous two cases, i.e., it is one of the followings:
\begin{enumerate}
\item the simple Lie groups $\Spin(n)$($n\geq 7$), $\sfE_6$, $\sfE_7$, $\sfE_8$, $\sfF_4$,
\item the real Grassmannians $\sfSO(p+q)/\sfSO(p)\times \sfSO(q)$($p\geq q\geq 3$, $p+q\geq 7$),
\item the following symmetric spaces
\begin{eqnarray*}
& \sfE_6/[\sfSp(4)/\set{\pm I}], \, \sfE_6/\sfSU(2)\cdot \sfSU(6), \, \sfE_7/[\sfSU(8)/\set{\pm I}], \, \sfE_7/\sfSO'(12)\cdot \sfSU(2) & \\
& \sfE_8/\sfSO'(16), \quad \sfE_8/\sfE_7\cdot \sfSU(2), \quad \sfF_4/\sfSp(3)\cdot \sfSU(2), \quad \sfG_2/\sfSO(4).& 
\end{eqnarray*}
\end{enumerate}
\end{enumerate}
\end{thm}

\begin{proof}
We show the classifications in cases (1) and (2) and then case (3) follows directly. From the first eigenvalue of Laplacian in \cite{Urakawa} (or see Tables in Appendix A), the following spaces are linearly unstable along conformal variations:
\begin{equation*}
\sfSU(n+1), \quad \sfSp(n), \quad \sfSU(2n)/\sfSp(n), \quad \sfSp(p+q)/\sfSp(p) \times\sfSp(q), \quad \sfE_6/\sfF_4, \quad \sfF_4/\Spin(9).
\end{equation*}
Note that the two exceptional examples in case (3) of Theorem \ref{thm:Koiso} for which the linear stability restricted to $\ker \dive \cap \ker \Tr$ is undecided are linearly unstable along conformal variations. From  case (2) in Theorem \ref{thm:Koiso} and Remark 4.2, there are two more examples of linearly unstable case which are $\sfSp(n)/\sfU(n)$ and $\sfSO(5)/\sfSO(3)\times \sfSO(2)$. They give all examples of linearly unstable case. 

An irreducible symmetric space of compact type has $\mu_{\text{fns}} = - 2\lambda$ if it is hermitian. It follows that a hermitian symmetric space of compact type is neutrally linearly stable if it is not unstable with respect to the Hilbert action. Hence it is one of the examples in case (2) except $\sfSU(n)/\sfSO(n)$.  For this latter example we have $\mu_L = - 2\lambda$ by \cite{KoisocompactSS} and so it is also neutrally linearly stable. 
\end{proof}

\medskip{}

\appendix

\section{First eigenvalues of Laplacian on functions and Lichnerowicz Laplacian on traceless transversal symmetric $2$-tensors}

We collect the first eigenvalues of Laplacian on functions for compact symmetric spaces, see also Tables A.1 and A.2 in \cite{Urakawa}. Note that the case of $\text{E III}(\sfE_6/\sfSO(10)\cdot \sfSO(2))$ is dropped there. For the Lichnerowicz Laplacian on symmetric $2$-tensors Koiso \cite{KoisocompactSS} showed that it is given by $\Delta_L = - C$ where $C$ is the Casimir operator of the $\sfK$ representation on the $2$-tensor bundle of $TM^\Cpx$, where $TM^\Cpx$ the complexified tangent bundle. It follows that the maximal eigenvalue $\mu_L$ of the Lichnerowicz Laplacian is bounded above by the smallest eigenvalue of the Casimir operator $C$ restricted to the traceless symmetric $2$-tensors, i.e., $\mu_L \leq -\min\set{L_i}$.  
 
In Tables \ref{tab:eigenvaluesgroups} and \ref{tab:eigenvaluesGK} we list the first eigenvalue $\mu_{\text{fns}}$ of Laplacian on functions and $\min\set{\lambda^{-1}L_i}$ for upper bound of the first eigenvalues $\mu_L$ of the Lichnerowicz Laplacian. An Einstein metric is called \emph{infinitesimal deformable}, or \emph{i-deformable} if there is a nonzero $h \in \ker\dive \cap \ker \Tr$ such that $\Delta_L h + 2\lambda h = 0$. In both tables ``i.d.", ``H.stable" and ``H.unstable" stand for ``i-deformable", ``stable with respect to Hilbert action", and ``unstable with respect to Hilbert action" respectively. The last column is for the linearly stability with respect to Perelman's $\nu$-entropy. ``l.stable", ``n.l.stable" and ``l.unstable" stand for ``linearly stable but not neutrally linearly stable", ``neutrally linearly stable" and ``linearly unstable".

The linear stability of a few symmetric spaces with respect to Perelman's $\nu$-entropy can be read off directly from the tables. For example, if both $-\lambda^{-1}\mu_{\text{fns}}$ and $\min\set{\lambda^{-1}L_i}$ are greater or equal to $2$, then $\sfG$ or $\sfG/\sfK$ is linearly stable. If $-\lambda^{-1} \mu_{\text{fns}} < 2$, then it is linearly unstable. However $\min\set{\lambda^{-1}L_i} < 2$ does not yield $\mu_L > -2\lambda$. For example, on $\sfSp(4)/\sfSp(3)\times \sfSp(1)$, we have $\lambda = 10$, $L_1 = 16$ and $L_2 = 36$. This space is stable with respect to the Hilbert action, i.e., $\mu_L < -2\lambda$. 

\medskip{}

\begin{table}[!htp]
\begin{tabular}{|l|l|c|c|l|l|}
\hline
type & $\sfG$ & $-\lambda^{-1} \mu_{\text{fns}}$ & $\min \lambda^{-1}L_i$ & H.stability & l.stability \\
\hline \hline
A$_n$ & $\sfSU(n+1)$\fts{($n\geq 2$)} & $\frac{2n(n+2)}{(n+1)^2}$ & $\frac{2n(n+2)}{(n+1)^2}$ & i.d. & l.unstable \\
\hline
\multirow{3}{*}{B$_n$} & $\Spin(5)$ & $\frac{5}{3}$ & $\frac{4}{3}$ & H.unstable & l.unstable \\
& $\Spin(7)$ & $\frac{21}{10}$ & $\frac{12}{5}$ &  H.stable & l.stable \\
& $\Spin(2n+1)$\fts{($n\geq 4$)} & $\frac{4n}{2n-1}$ & $\frac{4n}{2n-1}$ & H.stable & l.stable\\
\hline
C$_n$ & $\sfSp(n)$\fts{($n \geq 3$))} & $\frac{2n+1}{n+1}$ & $\frac{4n-1}{2(n+1)}$ & H.unstable & l.unstable \\
\hline
D$_n$ & $\Spin(2n)$\fts{($n\geq 4$)} & $\frac{2n-1}{n-1}$ & $\frac{2n-1}{n-1}$ & H.stable & l.stable \\
\hline
E$_6$ & $\sfE_6$ & $\frac{26}{9}$ &  $\frac{17}{6}$ & H.stable & l.stable \\
\hline
E$_7$ & $\sfE_7$ & $\frac{19}{6}$ & $3$ & H.stable & l.stable \\
\hline
E$_8$ & $\sfE_8$ & $4$ & $\frac{47}{15}$ & H.stable & l.stable \\
\hline
F$_4$ & $\sfF_4$ & $\frac{8}{3}$ & $\frac{8}{3}$ & H.stable & l.stable \\
\hline
G$_2$ & $\sfG_2$ & $2$ &  $2$ & H.stable & n.l.stable \\
\hline
\end{tabular}
\smallskip
\caption{The first eigenvalue $\mu_{\text{fns}}$ of Laplacian on functions and $\min\set{L_i}$ for the eigenvalues of the Casimir operator on symmetric $2$-tensors I: group type. Here $\lambda$ is the Ricci curvature.} \label{tab:eigenvaluesgroups}
\end{table}

\begin{center}
\begin{longtable}{|l|l|c|c|l|l|}
\hline
type & $\sfG/\sfK$ & $-\lambda^{-1}\mu_{\text{fns}}$ & $\min \lambda^{-1}L_i $ & H.stability & l.stability \\
\hline \hline
\endfirsthead

\multicolumn{6}{c}
{{\tablename\ \thetable{} -- \textsl{continued from previous page}}} \\
\hline 
type & $\sfG/\sfK$ & $-\lambda^{-1}\mu_{\text{fns}}$ & $\min \lambda^{-1}L_i$ & H.stability & l.stability \\
\hline \hline
\endhead

\multicolumn{6}{|r|}{{\textsl{Continued on next page}}} \\ \hline
\endfoot

\caption{The first eigenvalue $\mu_{\text{fns}}$ of Laplacian on functions and $\min\set{L_i}$  for the eigenvalues of the Casimir operator on symmetric $2$-tensors II: non-group type $\sfG /\sfK$. Here $\lambda$ is the Ricci curvature.} \label{tab:eigenvaluesGK} 

\endlastfoot

A I & $\sfSU(n)/\sfSO(n)$\footnotesize{($n \geq 3$)}  & $\frac{2(n-1)(n+2)}{n^2}$ & $2$ & i.d. & n.l.stable \\
\hline
\multirow{2}{*}{A II} & $\sfSU(4)/\sfSp(2)=\sph^5$ & $\frac{5}{4}$ & $3$ & H.stable & l.stable \\
& $\sfSU(2n)/\sfSp(n)$\fts{($n\geq 3$)} & $\frac{(2n+1)(n-1)}{n^2}$ & $2$ & i.d. & l.unstable \\
\hline
\multirow{2}{*}{A III} & $\frac{\sfU(p+1)}{\sfU(p)\times \sfU(1)} = \Cpx \mathrm{P}^p$ & $2$  &  $2$ & H.stable & n.l.stable \\ 
& $\frac{\sfU(p+q)}{\sfU(p) \times \sfU(q)}$\footnotesize{($p\geq q \geq 2$)} & $2$ & $2$ & i.d. & n.l.stable \\
\hline 
\multirow{6}{*}{B I, B II} 
& $\frac{\sfSO(2q+1)}{\sfSO(2q)} = \sph^{2q}$, \fts{($q\geq 1$)} & $\frac{2q}{2q-1}$ & $\frac{4q+2}{2q-1}$ & H.stable & l.stable \\
& $\frac{\sfSO(5)}{\sfSO(3)\times \sfSO(2)}$ & $2$ & $\frac{4}{3}$ & H.unstable & l.unstable \\
& $\frac{\sfSO(2p+3)}{\sfSO(2p+1)\times \sfSO(2)}$ \footnotesize{($p\geq 2$)} & $2$ & $\frac{8}{2p+1}$ & H.stable & n.l.stable \\
& $\frac{\sfSO(7)}{\sfSO(3)\times\sfSO(4)}$ & $\frac{12}{5}$ & $\frac{8}{5}$ & H.stable & l.stable \\
& $\frac{\sfSO(2q+3)}{\sfSO(3)\times \sfSO(2q)}$ \footnotesize{($q\geq 3$)} & $\frac{4q+6}{2q+1}$ & $\frac{8}{2q+1}$ & H.stable  & l.stable \\
& $\frac{\sfSO(2p+2q+1)}{\sfSO(2p+1)\times \sfSO(2q)}$ \footnotesize{($p\geq 2, q\geq 2$)} & $\frac{4p+4q+2}{2p+2q-1}$ & $\frac{8}{2p+2q-1}$ & H.stable & l.stable \\
\hline
C I & $\sfSp(n)/\sfU(n)$\fts{($n\geq 3$)} & $2$ & $\frac{2n}{n+1}$ & H.unstable & l.unstable \\
\hline 
\multirow{4}{*}{C II} &$\frac{\sfSp(2)}{\sfSp(1)\times \sfSp(1)} = \sph^4$ & $\frac{4}{3}$ & $\frac{10}{3}$ & H.stable & l.stable \\
& $\frac{\sfSp(3)}{\sfSp(2) \times \sfSp(1)}=\Qua \mathrm{P}^2$ & $\frac{3}{2}$ & $\frac{3}{2}$ & unknown & l.unstable \\
& $\frac{\sfSp(p+1)}{\sfSp(p) \times \sfSp(1)} = \Qua \mathrm{P}^p$\fts{($p\geq 3$)} & $\frac{2(p+1)}{p+2}$ & $\frac{2(p+1)}{p+2}$ & H.stable  & l.unstable \\
& $\frac{\sfSp(p+q)}{\sfSp(p) \times \sfSp(q)}$\fts{($p\geq q \geq 2$)} & $\frac{2(p+q)}{p+q+1}$ & $\frac{2(p+q)}{p+q+1}$ & unknown & l.unstable \\
\hline
\multirow{6}{*}{D I, D II} & $\frac{\sfSO(2p+2)}{\sfSO(2p+1)}= \sph^{2p+1}$ \fts{($p \geq 3$)} & $\frac{2p+1}{2p}$ & $\frac{2(p+1)}{p}$ & H.stable & l.stable \\
& $\frac{\sfSO(8)}{\sfSO(5)\times \sfSO(3)}$ & $\frac{5}{2}$ & $\frac{5}{2}$  & H.stable & l.stable \\
& $\frac{\sfSO(2q+2)}{\sfSO(2q)\times \sfSO(2)}$ \footnotesize{($q\geq 3$)} & $2$ & $2$  & H.stable & n.l.stable \\
& $\frac{\sfSO(2q)}{\sfSO(q)\times \sfSO(q)}$ \footnotesize{($q\geq 4$)} & $\frac{2q}{q-1}$  & $\frac{2q}{q-1}$  & H.stable & l.stable \\
& $\frac{\sfSO(2q+2)}{\sfSO(q+2)\times \sfSO(q)}$ \footnotesize{($q\geq 4$)} & $\frac{2q+2}{q}$ & $\frac{2q+2}{q}$ & H.stable & l.stable \\
& $\frac{\sfSO(2n)}{\sfSO(2n-q)\times \sfSO(q)}$ \footnotesize{($n-2\geq q\geq 3$)} & $\frac{2n}{n-1}$ &  $\frac{2n}{n-1}$ & H.stable  & l.stable \\
\hline 
D III & $\sfSO(2n)/\sfU(n)$\footnotesize{($n \geq 5$)} & $2$ & $2$  & H.stable & n.l.stable \\
\hline 
E I & $\sfE_6/[\sfSp(4)/\set{\pm I}]$ & $\frac{28}{9}$ & $3$  & H.stable & l.stable \\
\hline
E II & $\sfE_6/\sfSU(2) \cdot \sfSU(6)$ & $3$ & $3$ & H.stable & l.stable \\
\hline
E III & $\sfE_6/\sfSO(10)\cdot \sfSO(2)$ & $2$ & $2$  & H.stable & n.l.stable \\
\hline
E IV & $\sfE_6/\sfF_4$ & $\frac{13}{9}$ & $\frac{13}{9}$ & i.d. & l.unstable \\
\hline
E V & $\sfE_7/[\sfSU(8)/\set{\pm I}]$ & $\frac{10}{3}$ & $\frac{28}{9}$ & H.stable & l.stable \\
\hline
E VI & $\sfE_7/\sfSO'(12) \cdot \sfSU(2)$ & $\frac{28}{9}$ & $\frac{28}{9}$ & H.stable & l.stable \\
\hline
E VII & $\sfE_7/\sfE_6 \cdot \sfSO(2)$ & $2$ & $2$ & H.stable & n.l.stable \\
\hline
E VIII & $\sfE_8/\sfSO'(16)$ & $\frac{62}{15}$ & $\frac{16}{5}$ & H.stable & l.stable \\
\hline
E IX & $\sfE_8/\sfE_7 \cdot \sfSU(2)$ & $\frac{16}{5}$ & $\frac{16}{5}$ & H.stable & l.stable \\
\hline
F I & $\sfF_4/\sfSp(3) \cdot \sfSU(2)$ & $\frac{26}{9}$ & $\frac{26}{9}$ & H.stable & l.stable \\
\hline
F II & $\sfF_4/\Spin(9)$ & $\frac{4}{3}$ & $\frac{4}{3}$ & unknown & l.unstable \\
\hline
G & $\sfG_2/\sfSO(4)$ & $\frac{7}{3}$ & $\frac{7}{3}$ & H.stable  & l.stable \\
\hline
\end{longtable}
\end{center}

\medskip{}



\end{document}